\documentclass[10pt]{article}

\usepackage{inputenc}
\usepackage[english]{babel}
\usepackage{amssymb,amsmath,amsthm,
amscd,enumerate}
\usepackage{color}   

 \newenvironment{dedication} 
        {\vspace{6ex}\begin{quotation}\begin{center}\begin{em}}
        {\par\end{em}\end{center}\end{quotation}}

\newfont{\rams}{msbm10 scaled\magstep1}
\newfont{\ramss}{msbm10 scaled\magstep0}
\newfont{\iams}{msbm10}
\newfont{\gotic}{eufm10 scaled\magstep1}
\newfont{\bellap}{eusm10 scaled\magstep1}

\newcommand{\am}{A\!\Join^{{\!}f}{\!}J}
\newcommand{\q}{{J}}
\newcommand{\p}{{P}}

\newcommand{\da}{{A\!\Join^f\!\!\!J}}
\newcommand{\dasu}{{A\!\Join^{f'}\!\!\!J}}

\newcommand{\sss}{{\rm Spec}}
\newcommand{\spec}{{\rm Spec}}
\newcommand{\Jac}{{\rm Jac}}%

\newcommand{\mmm}{{\rm Max}}

\newcommand{\Ker}{{\rm Ker}}

\newcommand{\Hom}{{\rm Hom}}
\newcommand{\f}{\mathfrak}

\newcommand{\z}{{\ldots}}
\newcommand{\w}{{\setminus}}

\newcommand{\ude}{\mbox{\textsl{id}}} %

\swapnumbers
\newtheoremstyle{break}
  {9pt}
  {9pt}
  {\itshape}
  {}
  {\textsc}
  {.}
  {.7em}
  {}
\newtheoremstyle{break1}
  {9pt}
  {9pt}
  {}
  {}
  {\textsc}
  {.}
  {.7em}
  {}
\theoremstyle{break}
\newtheorem{thm}{ \textsc{Theorem}}[section]

\newtheorem{cor}[thm]{ \textsc{Corollary}}
\newtheorem{lem}[thm]{ \textsc{Lemma}}
\newtheorem{prop}[thm]{ \textsc{Proposition}}
\theoremstyle{break1}
\newtheorem{ex}[thm]{ \textsc{Example}}
\newtheorem{oss}[thm]{ \textsc{Remark}}
\theoremstyle{remark}

\title{New algebraic  properties of an
amalgamated algebra along an ideal \footnotetext{\hskip -15 pt
MSC: 13A15, 13B99,  14A05.} \footnotetext{\hskip -15 pt  Key
words: \it   idealization, pullback, Zariski topology, $D+M$
construction, Krull dimension, embedding dimension, Cohen-Macaulay, Gorenstein.}}
\author{Marco D'Anna (Universit\`a di Catania)
\and Carmelo A. Finocchiaro (Universit\`a degli Studi, Roma Tre)
\and Marco Fontana (Universit\`a degli Studi, Roma Tre)}

\begin{document}

\maketitle

\vskip -0.8cm
\begin{dedication}
\vspace*{0.5cm}{\small Dedicated to Alberto
Facchini on the occasion of his $60$th birthday}
\end{dedication}

\hfill{\footnotesize \sl }%

\smallskip

\begin{abstract}
Let $f:A \rightarrow B$ be a ring homomorphism and let $J$ be an
ideal of $B$. In this paper, we study the amalgamation of $A$ with
$B$ along $J$ with respect to $f$ (denoted by $\da$), a
construction that provides a general frame for studying  the
amalgamated duplication of a ring along an ideal, introduced by
D'Anna and Fontana in 2007, and other classical constructions
(such as the $A+ XB[X]$, the $A+ XB[\![X]\!]$ and the $D+M$
constructions). In particular, we completely describe the prime
spectrum of the amalgamation $\da$ and, when it is a local
Noetherian ring, we study its embedding dimension and when it
turns to be a Cohen-Macaulay ring or a Gorenstein ring.
\end{abstract}

\noindent{\scriptsize 
 The present version of the manuscript differs from the previous one,  posted on arXiv and published in Comm. Algebra {\bf 44} (2016), 1836--1851, for an appendix  
  --added at the end of the paper-- where we  observe that Proposition \ref{embdimineq}(2) and Theorem \ref{embdimeq} hold under the assumption,  not explicitly declared, that $B=f(A)+J$. Furthermore, in the same appendix, we provide the exact value for the embedding dimension of $\am$, also when $B\neq f(A)+J$, under the hypothesis that $J$ is finitely generated as an ideal  of the ring $f(A)+J$.
Finally, we also deleted Example 4.6. 
}
 
%
\bigskip

\section{Introduction}

Let $A$ and $B$ be commutative rings with unity, let $J$ be
an ideal of $B$ and let $f:A\rightarrow B$ be a ring
homomorphism. In this setting, we can consider the following
subring of $A\times B$:
$$
A\!\Join^f\!\! J := \{(a,f(a)+j) \mid  a \in A, \ j \in J \}
$$called {\it the amalgamation of $A$ with $B$ along $J$ with respect to $f$}.
This construction is a generalization of the amalgamated
duplication of a ring along an ideal (cf., for instance,
\cite{aung}, \cite{basataya},
\cite{c-j-k-m}, \cite{d'a-f-1}, \cite{d'a-f-2},  \cite{m-y} and \cite{sataya}). Moreover, several
classical constructions (such as the $A+XB[X]$, the
$A+XB[\![X]\!]$ and the $D+M$ constructions) can be studied as
particular cases of the amalgamation \cite[Examples 2.5 and
2.6]{dafifoproc} and other classical constructions, such as the
Nagata's idealization (cf.  \cite[Chapter VI,
Section 25]{H}, \cite[page 2]{N}), also called Fossum's trivial extension (cf.
\cite{Fos}), and the CPI extensions (in the sense of Boisen and
Sheldon \cite{bo}) are strictly related to it \cite[Example 2.7
and Remark 2.8]{dafifoproc}.

On the other hand, the amalgamation $\da$ is related to a
construction proposed by D.D. Anderson in \cite{a-06} and
motivated by a classical construction due to Dorroh \cite{do1},
concerning the embedding of a ring without identity in a ring with
identity. An ample introduction on the genesis  of the notion of
amalgamation is given in \cite[Section 2]{dafifoproc}.

One of the key tools for studying $\da$ is based on the fact that
the amalgamation can be studied in the frame of pullback
constructions \cite[Section 4]{dafifoproc} (for a systematic study
of this type of  constructions, cf. \cite{fa}, \cite{F},  \cite{ogoma}).
This point of view
allows us to deepen the study initiated in \cite{dafifoproc} and
continued in \cite{dafifoJPAA} and to provide an ample description
of various properties of $\da$, in connection with the properties
of $A$, $J$ and $f$. More precisely, in \cite{dafifoproc}, we
studied the basic properties of this construction (e.g., we
provided characterizations for $\da$ to be a Noetherian ring, an
integral domain, a reduced ring) and we characterized those
distinguished pullbacks that can be expressed as an amalgamation
and in \cite{dafifoJPAA} we investigated the Krull dimension of
$\da$. In this paper, we study in details its prime spectrum and,
when $\da$ is a local Noetherian ring, some of its invariants
(like the embedding dimension) and relevant properties (like
Cohen-Macaulyness and Gorensteinness).

In particular, after recalling (in Section 2) some basic
properties proved in \cite{dafifoproc}, needed in the present
paper, we provide a complete description of the prime spectrum of
$\da$ (Corollary \ref{spec}) and we characterize when $\da$ is a
local ring (Corollary \ref{localefarf}). In Section 3, we prove
some results on the extensions in $\da$ of ideals of $A$
(Proposition \ref{extension} and Lemma \ref{radical}), that we
will need in the sequel of the paper. In Sections 4 and 5, we
concentrate  our attention on the case when $\da$ is local; in
particular, we give bounds for its embedding dimension
(Proposition \ref{embdimineq}) and we produce classes of rings
$\da$ satisfying the upper or the lower bound (Proposition
\ref{edimlower} and Theorem \ref{embdimeq}). In the last section,
we study when $\da$ is a Cohen-Macaulay  or a Gorenstein ring
(Remarks \ref{regularsequence}, \ref{pr:7.6} and Proposition
\ref{pr:7.7}). Moreover, when $\da$ is Cohen-Macaulay, we
determine its multiplicity (Proposition \ref{pr:7.9}).

\section{The prime spectrum}
Before beginning a systematic study of the ring $\da$, we recall
from our introductory paper to the subject \cite{dafifoproc} the
notation that we will use in the present paper and some basic
properties of this construction.

\begin{prop}
{\rm \cite[Proposition 5.1]{dafifoproc}}\label{inizio}   Let $f:
A\rightarrow B$ be a ring homomorphism,  $\q$ an ideal of $B$  and
set $\da := \{ (a, f(a)+j) \mid a\in A, \ j \in J \}$.
\begin{enumerate}
 \item[\rm (1)]  Let  $\iota := \iota_{\!{_{A, f, J}}}: A\rightarrow \da$ be the natural
the ring homomorphism defined by $\iota(a)  := (a, f(a))$, for all
$a \in A$. The map $\iota$ is an embedding, making $\da$ a ring
extension of $A$.

\item[\rm (2)]  Let $I$ be an ideal of $A$ and set
$ I\!\!\Join^f\!\!\! J :=\{(i, f(i)+j) \mid i\in I, j \in J \}$. Then,
$I\!\!\Join^f\!\!\! J$  is an ideal of $\da$, the composition  of
canonical homomorphisms $A\stackrel{\iota}{\hookrightarrow}
\da\twoheadrightarrow (\da)/(I\!\!\Join^f\!\!\! J)$ is a surjective
ring homomorphism and its kernel coincides with  $I$.

\item[\rm (3)] Let $p_{\!{_A}}: \da \rightarrow A$ and
$p_{\!{_B}}:\da\rightarrow B$ be the natural projections of $\da
\subseteq A \times B$ {into} $A$ and $B$, respectively.  Then,
$p_{\!{_A}}$ is surjective
 and\, $\Ker(p_{\!{_A}})=\{0\}\times J$.
Moreover, $p_{\!{_B}}(\da)=f(A)+ J$ and $\Ker(p_{\!{_B}})=f^{-1}(\q)\times \{0\}$.

\item[\rm(4)] Let $\gamma:\da\rightarrow (f(A)+J)/J$
be the natural ring homomorphism, defined by $(a,f(a)+j)\mapsto f(a)+J$.
Then, $\gamma$ is surjective and
$\Ker(\gamma)=f^{-1}(J)\times J$.

\end{enumerate}
\end{prop}

Let $f:A\rightarrow B$   be a ring homomorphism and $\q$ an
ideal of $B$. In the present paper, we intend to further
investigate the algebraic properties of the ring $\da$, in
relation with those of $A,\ B,\ J$ and $f$. Recall that, in
\cite{dafifoproc}, we have shown that the ring $\da$ can be
represented as a pullback of natural ring homomorphisms and, using
the notion of ring retraction, we have characterized which type of
pullbacks are exactly of the form $\da$. In this paper, we will
make an extensive use of that idea for deepening the study of the
ring $\da$.


\medskip

\begin{oss}\label{de} \rm
(a)
Recall that, if $\alpha:A\rightarrow C,\,\,\,
\beta:B\rightarrow C$ are ring homomorphisms, the subring
$D:=\alpha\times_{_C}\beta:=\{(a,b)\in A\times B \mid
\alpha(a)=\beta(b)\}$ of $A\times B$ is called the
\textit{pullback} (or \textit{fiber product}) of $\alpha$ and
$\beta$.
We denote by   $p_{\!{_A}}$ (respectively, $p_{\!{_B}}$) the
restriction to $\alpha\times_{_C}\beta$ of the projection of
$A\times B$  onto $A$ (respectively, $B$).

The following statement is a straightforward consequence of the definitions:
Let $f:A\rightarrow B$ be a ring homomorphism and $\q$ be an ideal
of $B$. If $\pi:B\rightarrow B/J$ is the canonical projection and
$\breve f:=\pi\circ f$, then $\da=\breve f\times_{_{B/J}}\pi$.

(b)
Recall that a ring homomorphism $r:B\rightarrow A$
is called {\it a ring retraction} if there exists an (injective)
ring homomorphism $i:A\rightarrow B$ such that $r\circ i=\ude_A$.
In this case, we say also that $A$ is {\it a retract of} $B$.
By {\cite[Remark 4.6]{dafifoproc}}, with the previous notation, we have that
 $A$ is a retract of $\da$ and the map
$p_{\!{_A}}:\da\rightarrow A$, defined in Proposition
\ref{inizio}(3), is a ring retraction. In fact, we have
$p_{\!{_A}}\circ \iota= \ude_{\!{_A}}$, where $\iota$ is the ring
embedding of $A$ into $\da$ (Proposition \ref{inizio}(1)).

(c) The pullbacks of the form $\da$ form a distinguished subclass of
the class of pullbacks of ring homomorphisms, as described in
{\rm \cite[Proposition 4.7]{dafifoproc}}.
Let $A,B,C,\alpha,\beta,p_{\!{_A}},p_{\!{_B}}$ be as in
(a). Then,
 $p_{\!{_A}}: D \ (= \alpha\times_{_C}\beta) \rightarrow A$ is a ring retraction
if and only if
 there exists an ideal $J$ of $B$ and a ring homomorphism
$f:A\rightarrow B$ such that $D\cong \da$.

(d) Note that, using the notation in (a), we are not making any assumption on the ring
homomorphism $\alpha: A \rightarrow C$ nor on the homomorphism
$\breve f:=\pi\circ f: A \rightarrow B/J$. In \cite{aam} the
authors consider a new construction, called connected sum of local
rings, obtained by taking a quotient of a pullback for which both
the homomorphisms $\alpha$ and $\beta$ are surjective. A
particular case of this type of pullback is the amalgamated
duplication $A\!\Join\!I$, where $A$ is a local ring and $I$ an
ideal of $A$ (see \cite{d'a-f-1} and \cite{d'a-f-2}).

(e) Note that
the amalgamation $\da$, even in the local case, may not be fully
re-conducted to a pullback for which both the homomorphisms
$\alpha$ and $\beta$ are surjective. However, changing the data,
and considering $B':= f(A) +J$, $J$ as an ideal of $B'$, and  $f'
: A \rightarrow B'$ acting as $f$, it is easy to see that $\da =
\dasu$ and $\dasu$ is a pullback of $\pi': B' \rightarrow B'/J$
and $\breve{f'}:=\pi'\circ f' : A \rightarrow B'/J$ (i.e., $\dasu
= \breve{f'} \times_{B'/J} \pi'$),
which are now both surjective. But, this is only apparently a
simplification of the given construction, since the problem of
studying  $\da$ from the data $A, B, J, f$ is transformed into the
problem of studying $\dasu$ and the ring inclusion $f(A)+J
\hookrightarrow B$, and the last problem presents the same level
of complexity of a direct investigation of the given construction
(see for instance \cite[Section 5]{dafifoproc} and \cite[Section
4]{dafifoJPAA}).

\end{oss}


Let  $f: A\rightarrow B$ be a ring homomorphism, and set
$X:=\sss(A), \ Y:=\sss(B)$. Recall that $f^*:Y\rightarrow X$
denotes the continuous map (with respect to the Zariski
topologies)  naturally associated to $f$  (i.e.,  $f^*(Q):=
f^{-1}(Q)$ for all $Q\in Y$).  Let  $S$ be a subset of $A$. Then,
as usual,   $V_X(S)$, or simply $V(S)$, if no confusion can arise,
denotes the closed subspace of $X$, consisting of all prime ideals
of $A$ containing $S$.

In the next lemma we recall the notation and some basic properties
of pullback constructions that we will use in the present paper.
We refer to the paper by Fontana \cite{F}, since the subsequent
work on pullbacks by Facchini \cite{fa} and, in the Noetherian
setting, by Ogoma \cite{ogoma} is not relevant to our study.

\begin{lem}\label{fon} \rm \cite[Theorem 1.4]{F}  \it
With the notation of Remark \ref{de} (a),
set $X:=\emph{\sss}(A)$,
\ $Y:=\emph{\sss}(B),\ Z:=\emph{\sss}(C)$, and
$W:=\emph{\sss}(D)$. Assume that $\beta$ is surjective. Then, the
following statements hold.
\begin{enumerate}
\item [\rm (1)] If $H \in W\w V(\Ker(p_{\!{_A}}))$, then there is
a unique prime ideal $Q$ of $B$ such that $p_{\!{_B}}^{-1}(Q)= H$.
Moreover, $Q\in Y\w V(\Ker(\beta))$ and $D_H\cong B_Q$, under the
canonical homomorphism induced by $p_{\!{_B}}$.
\item [\rm (2)] The continuous map $p_{\!{_A}}^*$ is a closed embedding of $X$ into $W$.
Thus $X$ is homeomorphic to its image, $V(\Ker(p_{\!{_A}}))$, under $p_{\!{_A}}^*$.
\item [\rm (3)] The restriction of the continuous map $p_{\!{_B}}^*$ to $Y\w V(\Ker(\beta))$
is an homeo\-morphism of $Y\w V(\Ker(\beta))$ with $W\w
V(\Ker(p_{\!{_A}}))$  (hence,  \emph{a fortiori}, it is an
isomorphism of partially ordered sets).
\end{enumerate}
In particular, the prime ideals of $D$ are of the type
$p_{\!{_A}}^{-1}(P)$ or $p_{\!{_B}}^{-1}(Q)$, where $P$ is any
prime ideal of $A$ and $Q$ is a prime ideal of $B$, with
$Q\nsupseteq \Ker(\beta)$.
\end{lem}

The following corollary is a direct consequence of Lemma \ref{fon}.

\begin{cor}\label{pmax}
With the notation of Remark \ref{de} (a), assume that $\beta$ is
surjective. Let $H$ be a prime ideal of $D \
(=\alpha\times_{_C}\beta)$.
\begin{enumerate}
\item [\rm(1)] Assume that $H$ contains $\Ker(p_{\!{_A}})$. Let
$P$ be the only prime ideal of $A$ such that $H=p_{\!{_A}}^*(P)$
(Lemma \ref{fon}(2)). Then,  $H$ is a maximal ideal of $D$ if
and only if $P$ is a maximal ideal of $A$.
\item [\rm(2)] Assume that $H$ does not contain $\Ker(p_{\!{_A}})$. Let
$Q$ be the only prime ideal of $B$ ($Q \notin V(\Ker(\beta))$)
such that $p_{\!{_B}}^*(Q)=H$ (Lemma \ref{fon}(1)). Then, $H$ is
a maximal ideal of $D$ if and only if $Q$ is a maximal ideal of
$B$.
\item [\rm(3)]
$D \ (=\alpha\!\times_{_C}\!\beta)$ is a local
ring if and only if $A$ is a local ring and ${\rm
Ker}(\beta) $
is contained in the Jacobson radical ${\rm Jac}(B)$. In particular, if $A$ and $B$
are local rings, then $D$ is a local ring.  Moreover, if $D$ is a
local ring and $M$ is the only maximal ideal of $A$, then
 $\{p_{\!{_A}}^{-1}(M)\} =
\mmm(D)$.
\end{enumerate}
\end{cor}

As a consequence of the previous results we can now easily describe the structure of
the prime spectrum of the ring $\da$. The details of the proof are omitted.

\begin{cor}\label{spec}
 With the notation of Proposition \ref{inizio},  set $X:=\sss(A)$, $Y:=\sss(B)$,  and $W:=\sss(\da)$,
$J_0 :=\{0\}\times J \ (\subseteq \da)$, and $J_1:=f^{-1}(J)\times \{0\}$. For all $P \in X$ and $Q \in Y$, set:
$$
\begin{array}{rl}
P^{\prime_{_{\! f}}}:= &\hskip -4pt  P\!\Join^f\!\! J :=\{(p, f(p)+j ) \mid p\in \p,\ j\in J \}\,,\\
\overline Q^{_{_{ f}}}:=&\hskip -4pt \{(a, f(a)+j) \mid a\in A,\  j \in J,\ f(a)+ j \in Q\}\,.
\end{array}
$$
Then, the following statements hold.
\begin{enumerate}
\item [\rm (1)] The map $P \mapsto P^{\prime_{_{\! f}}}$
establishes a closed embedding of $X$ into  $W$,
so its image, which coincides with $V(J_0)$, is homeomorphic to
$X$.
\item [\rm (2)] The map $Q \mapsto\overline Q^{_{_{f}}} $ is {{a}} homeomorphism of $Y\w V(J)$ onto
$W\w V(J_0)$.

\item [\rm (3)] The prime ideals of $\da$ are of the type\
$P^{\prime_{_{\! f}}}$ or \ $\overline Q^{_{_{ f}}}$\!,\,  for
$P$ varying in $ X$ and $Q$ in $Y\w V(\q)$.

\item [\rm (4)]
$W=V(J_0)\cup V(J_1)$
and the set $V(J_0)\cap V(J_1)$ is homeomorphic to $\sss((f(A)+J)/J)$,
via the continuous map associated to the natural ring homomorphism
$\gamma:\da\rightarrow (f(A)+J)/J$, $(a,f(a)+j)\mapsto f(a)+J$.
In particular, we have that the closed  subspace $V(J_0)\cap V(J_1) $ of $W$ is
homeomorphic to the closed subspace $V(J)$ of $Y(=\sss(B))$, when $f$ is surjective.
\end{enumerate}
\end{cor}

The following example provides a geometrical illustration of some of the material presented above.
\begin{ex}
Let $K$ be an algebraically closed field and $X, Y$ two
indeterminates over $K$.  Set $A := K[X, Y]$, $B:=  K[X]$ and $f:
K[X, Y] \rightarrow   K[X] $ defined by $Y \mapsto 0$ and $X
\mapsto X$.  Let $J:= XK[X]$. We want to study the ring
$K[X,Y]\!\Join^{f}\!\!J$
(note that, from a
geometrical point of view, $f^*$ determines the inclusion of the
line defined by the equation $Y=0$ into the affine space $\mathbb
A_K^2$.)

 According to the notation of Corollary \ref{spec}, we
have $V(J_1)\cong \sss(K[Y])$. Moreover, the projection
$p_{\!{_B}}$ of $\da$ into $B$ is surjective, since $f$ is
surjective, and its kernel is $J_1$ (see Proposition
\ref{inizio}). Thus $\sss(\da/J_1)\cong V(J_1)\cong
\sss(B)=\sss(K[X])$. We have also $V(J_1)\cap V(J_2)\cong
\sss(B/J)=\sss(K)$, by Corollary \ref{spec}(4). Then, $\da$ is the
coordinate ring of the union of a plane  (i.e., $\sss(K[X, Y])$)
and a line (i.e., $\sss(K[X])$) with one common point (i.e.,
$\sss(K)$). Note that, in this case, the ring $\da$ can be also
presented by a quotient of a polynomial ring. Indeed, since $f$ is
surjective and  $B/J\cong K$, by a standard argument we easily
obtain that $\da$ is isomorphic to $K[X,Y,Z]/(ZX,YZ)$.
\end{ex}

 If we specialize  Corollary \ref{pmax}
to the case of the construction $\da$, then we
obtain the following:

 \begin{cor}\label{localefarf}
 We preserve the notation of
  Corollary \ref{spec}.
\begin{enumerate}
\item [\rm(1)] Let $P \in X$. Then, $P^{\prime_{_{\! f}}}$
 is a maximal ideal of $\da$ if and only if $P$ is a maximal ideal of $A$.

\item [\rm(2)] Let $Q$ be a prime ideal of $B$ not containing $J$. Then,
$\overline Q^{_{_{f}}}$ is a maximal ideal of $\da$ if and only if
$Q$ is a maximal ideal of $B$.

 In particular,
$ \mmm(\da)\!=\!\{P^{{\prime_{_{\! f}}}} \mid  P \in\mmm(A)\} \cup
\{\overline Q^{_{_{f}}} \mid Q \in \mmm(B)\w V(\q)\}.$
  \end{enumerate}
  \begin{enumerate}

\item[\rm(3)]
  $\da$ is a local ring
  if and only if
$A$ is a local ring and $\q\subseteq {\rm Jac}(B)$.

In particular, if $M$ is the unique maximal ideal of $A$, then
$M^{\prime_{_{\! f}}}= M \!\!\Join^f\!\!\!J$ is the unique maximal
ideal of $\da$.
\end{enumerate}
\end{cor}
The following result, whose proof is straightforward, provides a
description of the minimal prime ideals of $\da$.
\begin{cor}
With the  notation of Corollary \ref{spec}, set
$$
\mathcal X:=\mathcal X_{(f,J)}:=\bigcup_{Q\in \spec(B)\setminus V(J)}V(f^{-1}(Q+J))
$$
The following properties hold.
\begin{enumerate}[\rm (1)]
\item The map $Q\mapsto \overline Q^{_{_{f}}}$
establishes a homeomorphism of ${\rm Min}(B)\setminus V(J)$ with
${\rm Min}(\da)\setminus V(J_0)$.
\item The map $P\mapsto P^{{\prime_{_{\! f}}}}$ establishes a
homeomorphism of ${\rm Min}(A)\setminus \mathcal X$ with\,\,\, ${\rm Min}(\da)\cap V(J_0)$.
\end{enumerate}
\end{cor}
After describing the topological and ordering properties of the
prime spectrum of the ring $\da$, we now describe the
localizations of $\da$ at each of its prime ideals.
\begin{prop}
With the notation of Proposition \ref{inizio} and Corollary \ref{spec}, the following statements hold.
\begin{enumerate}[\rm (1)]
\item For any prime ideal $Q\in Y\setminus V(J)$, the ring
$(\da)_{\overline Q^{_{_{f}}}}$ is canonically isomorphic to
$B_Q$.
\item For any prime ideal $P\in X\setminus V(f^{-1}(J))$,
the localization $(\da)_{P^{{\prime_{_{\! f}}}}}$ is canonically
isomorphic to $A_P$.
\item Let $P$ be a prime ideal of $A$ containing $f^{-1}(J)$.
Consider the multiplicative subset $S:=S_{(f,P,J)}:=f(A\setminus
P)+J$ of $B$ and set $B_S:=S^{-1}B$ and $J_S:=S^{-1}J$. If
$f_P:A_P\longrightarrow B_S$ is the ring homomorphism induced by
$f$, then the ring $(\da)_{P^{{\prime_{_{\! f}}}}}$ is canonically
isomorphic to $A_P\Join^{f_P}J_S$.
\end{enumerate}
\end{prop}
\noindent \textsc{Proof.} Keeping in mind the fiber product
structure of $\da$, (1) follows from Lemma \ref{fon} and (2) is
straightforward. From the last part of Remark \ref{de}(a) we infer
that, if $\breve{f_P}:A_P\longrightarrow B_S/J_S$ is the ring
homomorphism induced by $f_P$ and if $\pi_{(P)}:B_S\longrightarrow
B_S/J_S$ is the canonical projection, then $A_P\Join^{f_P}J_S$ is
isomorphic to the fiber product
$\breve{f_P}\times_{B_S/J_S}\pi_{(P)}$. Moreover, it is easily
verified that $p_A(\da\setminus P^{{\prime_{_{\! f}}}})=A\setminus
P$ and $p_B(\da\setminus P^{{\prime_{_{\! f}}}})= S$. Then
statement (3) follows from  \cite[Proposition 1.9]{F}.\hfill$\Box$


\section{Extension of ideals of $A$ to $\da$}

In this section we pursue the study of the ideal-theoretic
structure of the amalgamation $\da$.

\begin{prop}\label{extension}
We preserve the notation of Proposition \ref{inizio} and Corollary
\ref{spec}. The following properties hold.
\begin{enumerate}[\rm (1)]
\item If $I$ (respectively, $H$) is an ideal of $A$ (respectively, of $ f(A)+J$)
such that $f(I)J\subseteq H\subseteq J$, then
$
I\!\Join^f\!\!H:= \{(i,f(i)+h)\mid i\in I,\, h\in H\}
$
is an ideal of $\da$.
\item If $I$ is an ideal of $A$, then the extension $I(\da)$ of $I$ to $\da$ coincides with
$
I\!\Join^f\!\!(f(I)B)J:= \{(i,f(i)+\beta)\mid i\in I,\, \beta\in
(f(I)B)J\}\,.$
\item If $I$ is an ideal of $A$ such that $f(I)B=B$, then
$
I(\da)=I^{\prime_f}= \{(i, f(i)+j) \mid   i\in I,\,   j \in J\} =
I\!\Join^f\!\!J\,.
$
\end{enumerate}
\end{prop}

\noindent \textsc{Proof.} (1) is straightforward. (2). Set
$I_0:=I\!\Join^f\!\!(f(I)B)J$. By applying (1) to $H:=( f(I)B)J$,
it follows that $I_0$ is an ideal of $\da$ and, by definition,
$I_0\supseteq\iota(I) \ (= \{(i,f(i)) \mid i\in I\})$. Now, let
$L$ be an ideal of $\da$ containing $\iota(I)$, and let $(i, f(i)
+ \beta)\in I_0$ (where $i\in I$ and $\beta \in (f(I)B)J$).
Therefore, we can find $\alpha_1, \alpha_2,\z,\alpha_n\in I$,
 $b_1,b_2,\z,b_n\in J$ such that $\beta=\sum_{k=1}^n f(\alpha_k)b_k$.
Since, $(i,f(i)),(\alpha_1,f(\alpha_1)), $ $
(\alpha_2,f(\alpha_2)), \z,$ $(\alpha_n,f(\alpha_n))\in
\iota(I)\subseteq L$, then
$$
 (i,f(i)+ \beta)=(i,f(i))+\sum_{k=1}^n
(\alpha_k,f(\alpha_k))(0,b_k)\in L\,.
$$
and so $I_0\subseteq L$. The proof  of (2) is now complete. (3)
follows immediately from (2). \hfill$\Box$

\begin{cor}\label{radical} Let $A$ be a local ring with
maximal ideal $M$, let $f:A\rightarrow B$ be a ring homomorphism,
and $J$ be an ideal of $B$ such that $f^{-1}(Q)\neq M$, for each
$Q\in \spec(B)\setminus V(J)$. If $I$ is an ideal of $A$ whose
radical is $M$, then the radical of $I(\da)$ is $M^{\prime_f} \ (=
M\!\Join^f\!\!J)$.
\end{cor}
\noindent \textsc{Proof.}  Suppose that $P$ is a prime ideal of
$A$ such that $P^{\prime_f}\supseteq I(\da)$. It follows
immediately that $I\subseteq P$ and thus $P=M$, by assumption.
Suppose now that $I(\da)\subseteq \overline{Q}^f$, for some $Q\in
\spec(B)\setminus V(J)$. From Proposition \ref{extension}(2) and
the definition of $\overline{Q}^f$, we deduce that
$(f(I)B)J\subseteq Q$ and, in particular, $f(I)\subseteq Q$, i.e.,
$I\subseteq f^{-1}(Q)$; therefore, by assumption, $f^{-1}(Q)= M$,
which is a contradiction. This means that the unique prime ideal
of $\da$ containing $I(\da)$ is $M^{\prime_f}$. \hfill$\Box$

\begin{oss}
Notice that, in case $J$ is finitely generated as $A$-module
and it is contained in the Jacobson radical of $B$, for every
prime $Q$ of $B$ not containing $J$, we have $f^{-1}(Q)\neq M$. In
fact, if we had $f^{-1}(Q)= M$, we would have $f(M) \subseteq Q$,
that implies $J/QJ$ is finite dimensional as $A/M$-vector space;
now, $J \not\subset Q$ and $Q$ is a prime ideal, so if $j \in J
\setminus Q$ then $j^n \in J\setminus Q$, for every integer $n\geq
1$ and, since $J \subseteq {\rm Jac}(B)$, it is not difficult to
check that the images of the elements $j, ,j^2, \dots, j^n$ in
$J/QJ$ are linearly independent over $A/M$ for any $n$, that is a
contradiction.

In particular, if $J$ is finitely generated as $A$-module and it
is contained in the Jacobson radical of $B$, the extension in
$\da$ of any $M$-primary ideal of $A$ is $M
\!\!\Join^f\!\!\!J$-primary.
\end{oss}

\section{The embedding dimension of $\da$}
Let $A$ be a ring and $I$ be an ideal of $A$. If $I$ is finitely
generated,  we denote,  as usual, by $\nu(I)$ the minimum number
of generators   of  the ideal  $I$. Assume that $A$ is  a local
ring and that $M$ is its maximal ideal. Set $\boldsymbol{k}:=
A/M$.  If we suppose  that $M$ is finitely generated, we call the
{\it embedding dimension of} $A$ the natural number
$$
{\rm embdim}(A):= \nu(M)=\dim_{\boldsymbol{k}}(M/M^2)\,.
$$
We give next some bounds for the embedding dimension of $\da$,
when this ring is local with finitely generated maximal ideal.

\begin{prop}\label{embdimineq}
We preserve the notation of Proposition \ref{inizio}. Assume that
$A$ is a local ring with maximal ideal $M$ and that the ideal $J$
is contained in  the Jacobson radical $\Jac(B)$. The following statements hold.
\begin{enumerate}[\rm (1)]
\item  If $\da$ has finitely generated maximal ideal, then $A$ has also finitely generated maximal ideal
and
$$
{\rm embdim}(A)\leq {\rm embdim}(\da).
$$
\item{\hskip -5pt \footnote{  \, see the Appendix}}  If $A$ has  finitely generated maximal ideal and $J$ is
finitely generated, then $\da$ has finitely generated maximal
ideal and
$$
 {\rm embdim}(\da)\leq {\rm embdim}(A) +\nu(J).
$$
\end{enumerate}
\end{prop}
\noindent \textsc{Proof.} By using Corollary \ref{localefarf}(3),
it follows that $\da$ is a  local ring with maximal ideal $
M^{\prime_f}:=M\!\Join^f\!\!J:=\{(m, f(m)+j)\mid
m\in M,\  j\in J\}. $\\
(1) It suffices to note that, if $\{{\bf x}_1, {\bf x}_2,\z,{\bf
x}_n\}$ is a finite set of generators of $M^{\prime_f}$, then
$\{p_A({\bf
x}_i)\mid  i=1, 2,\z,n\}$ is a finite set of generators of $M$.\\
(2) Let \ $m_1, m_2, \z,m_r\in M$ \ and \ $j_1, j_2, \z,j_s\in J$
\ be elements such that \linebreak
 $M =(m_1, m_2,
\z,m_r) $ and $ J=(j_1, j_2, \z,j_s)$,  with $\nu(M)=r$ and
$\nu(J)=s$. It follows immediately that $ \{(m_\lambda,
f(m_\lambda)); \,  (0, j_\mu) \mid 1\leq \lambda \leq r, \, 1\leq
\mu \leq s\} $ is a set of generators of $M\!\Join^f\!J$.
Therefore,  ${\rm embdim}(\da)\leq{\rm embdim}(A)+\nu(J)$.
\hfill$\Box$

\medskip

In the next example we will provide a ring homomorphism
$f:A\rightarrow B$ and an ideal $J\neq (0)$ of $B$ such that ${\rm
embdim}(A)={\rm embdim}(\da)<{\rm embdim}(A)+\nu(J)$.

\begin{ex}
Let $p$ be a prime number, $T$ be an indeterminate over $\mathbb
Q$, and set $A:=\mathbb Z_{(p)}, B:=\mathbb Q[\![T]\!]$, $ J:=TB$.
By \cite[Example 2.6]{dafifoproc}, the ring $S:=A+TB$ is naturally
isomorphic to $\da$, where $\iota:A\rightarrow B$ is the
inclusion. It is easy to see that $S$ is a 2-dimensional valuation
domain whose maximal ideal $N \ (:= p\mathbb Z_{(p)} +TB) $ is
principal (namely, $N =pS$). It follows that ${\rm embdim}(A)={\rm
embdim}(\da)=1<{\rm embdim}(A)+\nu(J)=2$.
\end{ex}

The previous example is a particular case of the following result.

\begin{prop}\label{edimlower}
We preserve the notation of Proposition \ref{inizio} and Corollary
\ref{spec}. Assume that $A$ is a local ring with finitely
generated maximal ideal $M$ satisfying the property $f(M)B=B$.
Then, for every ideal $J$ of $B$ contained in the Jacobson radical
of $B$, the amalgamation $\da$ is a local ring with finitely
generated maximal ideal, and
$$
{\rm embdim}(A)={\rm embdim}(\da)\,.
$$
\end{prop}
\noindent \textsc{Proof.} Let $r:={\rm embdim}(A)$ and let $\{m_1,
m_2, \z,m_r\}$ be a minimal set of generators for $M$. By
Corollary \ref{localefarf}(3), $\da$ is a local ring with maximal
ideal $M^{\prime_f}:=\{(m,f(m)+j)\mid M\in M, j\in J\}$ and,
applying Proposition \ref{extension}(3), we get the equality
$M^{\prime_f}=M(\da)$. It follows immediately that
$\{(m_1,f(m_1)),(m_2,f(m_2)),\z,(m_r,f(m_r))\}$ is a finite set of
generators for $M^{\prime_f}$ and, thus, ${\rm embdim}(\da)\leq
r:={\rm embdim}(A)$. Now, the conclusion is an immediate
consequence of Proposition \ref{embdimineq}(1).
\hfill$\Box$\\

The next result will provide a relevant class of rings obtained by
amalgamation satisfying the equality ${\rm embdim}(\da)={\rm
embdim}(A)+\nu(J)$.

\begin{thm}{\hskip -6pt \footnote{ { \, see the Appendix}}} \label{embdimeq}
We preserve the notation of Proposition \ref{inizio}. Suppose that
$A$ is a local ring with finitely generated maximal ideal $M$, and
that $J$ is a finitely generated ideal of $B$. If $f(M)B \subseteq
\Jac(B)$ and $J\subseteq \Jac(B)$, then $\da$ is a local ring with
finitely generated maximal ideal, and
$$
{\rm embdim}(\da)={\rm embdim}(A)+\nu(J)\,.
$$
\end{thm}

\noindent \textsc{Proof.} Let $\{m_1, m_2, \z,m_r\}\subseteq M$,
and  $\{j_1, j_2, \z,j_s\}\subseteq J$ be sets of generators of
$M$ and $J$, respectively, such that $\nu(M)=r$ and $\nu(J)=s$. By
Proposition \ref{embdimineq} and its proof it follows immediately
the inequality ${\rm embdim}(\da)\leq{\rm embdim}(A)+\nu(J)$ and,
more precisely, that $G':=\{(m_\lambda,f(m_\lambda)); \,
(0,j_\mu)\mid 1\leq \lambda\leq r,1\leq\mu\leq s\}$ is a set of
generators of the maximal ideal $M' :=M^{\prime_f} =
M\!\Join^f\!\!J$ of $\da$. Notice that $\boldsymbol{k}$, the
residue field of $A$, coincide with the residue field of $\da$
(see Proposition \ref{inizio}(2)). Then, to get the equality ${\rm
embdim}(\da)={\rm embdim}(A)+\nu(J)$ it suffices to show that the
image $\overline{G'}$ of $G'$ in $M'/M'^2$ is a basis of $M'/M'^2$
as a $\boldsymbol{k}-$vector space. Obviously, it is enough to
check that $\overline{G'}$ is linearly independent. Pick $a_1,
a_2, \z,a_r,\alpha_1, \alpha_2, \z,\alpha_s\in A$ such that
$$
\sum_{\lambda=1}^r[a_\lambda]_{M}[(m_\lambda,f(m_\lambda))]_{M'^2}+\sum_
{\mu=1}^s[\alpha_\mu]_{M}[(0,j_\mu)]_{M'^2}=0 \,.  \quad (\star)
$$
In other words, we have
$$
\left(\sum_{\lambda=1}^ra_\lambda m_\lambda,\; \sum_{\lambda=1}^r
f(a_\lambda m_\lambda)+\sum_{\mu=1}^s f(\alpha_\mu)j_\mu\right)\in
M'^2
$$
and, in particular, $\sum_{\lambda=1}^ra_\lambda m_\lambda\in
M^2$. Since $r=\nu(M)$, it is easy to see that $a_\lambda\in M$,
for every $\lambda=1,2, \z,r$.
 Thus, by $(\star)$, we have $\sum_
{\mu=1}^s[\alpha_\mu]_{M}[(0,j_\mu)]_{M'^2}$ $=0 $ and so
$$
\left(0,\; \sum_{\mu=1}^s f(\alpha_\mu)j_\mu\right)\in M'^2 \,.
$$
This means that $\left(0,\; \sum_{\mu=1}^s
f(\alpha_\mu)j_\mu\right)$ is a finite sum of elements of the form
$(m,\, f(m)+j)(n,\, f(n)+\ell)$, where $m,n \in M$ and $ j, \ell
\in J$. Then, an easy computation shows that  \ $\sum_{\mu=1}^s
f(\alpha_\mu)j_\mu$ \ is a finite sum of elements of the form \
$f(m)\ell+f(n)j+j\ell$ \ and thus the element  \
$b:=\sum_{\mu=1}^s f(\alpha_\mu)j_\mu\in  (f(M) B)J+J^2\subseteq
\Jac(B)J$. Suppose, by contradiction, that some coefficient
$\alpha_\mu\in A\setminus
 M$, say $\alpha_1$, and let $\beta_1$ denote the
inverse of $f(\alpha_1)$ in $B$.  Then $\beta_1 b \in \Jac(B)J$,
and thus there are elements $l_1, l_2, \z,l_s\in \Jac(B)$ such
that
$$ \beta_1 b =j_1+\sum_{\mu=2}^s \beta_1 f(\alpha_\mu)j_\mu=\sum_
{\mu=1}^sl_\mu j_\mu.$$
 This shows that $(1-l_1)j_1\in
(j_2,\z,j_s)B$, and hence, keeping in mind that $l_1\in \Jac(B)$,
we have $j_1\in (j_2,\z,j_s)B$, a contradiction. Thus $\alpha_\mu
\in M$ for $\mu=1, 2, \z,s$. The proof is now complete.
\hfill$\Box$

\medskip

As an application we obtain the following.

\begin{cor}
Let $A$ be a local ring with finitely generated maximal ideal, and
let $I$ be a finitely generated proper ideal of $A$. Then, the
duplicated amalgamation $A\!\Join\! I$ of $A$ along $I$ is a local
ring with finitely generated maximal ideal, and furthermore
$
{\rm embdim}(A\!\Join\! I)={\rm embdim}(A)+\nu(I)
$.
\end{cor}

\noindent
\textsc{Proof.} Apply \cite[Example 2.4]{dafifoproc} and Proposition \ref{embdimeq}.\hfill$\Box$\\\\
\section{Cohen--Macaulay and Gorenstein properties for the ring $\da$}

In this section, assuming that $\da$ is local and Noetherian, we
investigate the problem of when $\da$ is a Cohen-Macaulay (briefly CM) ring or
a Gorenstein ring. Moreover, when $\da$ is Cohen-Macaulay, we
determine its multiplicity.

\textsc{Notation and assumptions.}

In the following (unless explicitly stated to the contrary), we
assume that:
\begin{itemize}
	\item $f:A\rightarrow B$ is a ring homomorphism;
	\item $A$ is
	Noetherian, local, with maximal ideal $M$;
	\item $J$ is an ideal of
	$B$ contained in the Jacobson radical ${\rm Jac}(B)$ of $B$;
	\item $J$ is finitely generated as an $A$-module.
\end{itemize}

 In this situation
(by \cite[Proposition 5.7]{dafifoproc} and by Corollary
\ref{localefarf}(3)) we know that the amalgamated algebra $\da$ is
a Noetherian local ring, with maximal ideal $M^{\prime_f} $.
Moreover, the canonical map $\iota:A\rightarrow \da$ is a finite
ring embedding, since $J$ is finitely generated as an $A$-module
\cite[Proposition 5.7]{dafifoproc}, and thus $\dim(A)=\dim(\da)$.
Moreover Ann$(\da)=(0)$, hence the dimension of $\da$ as
$A$--module (or, equivalently, dim$(A/{\rm Ann}(\da))$, since
$\da$ is a finite $A$--module) equals the Krull dimension of
$\da$.

\begin{oss}\label{regularsequence}
We observe that, under the previous assumptions, $\da$ is a CM
ring if and only if it is a CM $A$--module if and only if $J$ is a
maximal CM $A$-module.

As a matter of fact, since the embedding $\iota:A\hookrightarrow
\da$  is finite, by \cite[Exercise 1.2.26(b)]{b-h} we have ${\rm
depth}_A(\da)={\rm depth}(\da)$, and thus, by the discussion
above,
 $\da$ is a CM ring if
and only if $\da$ is a CM $A$--module. Since $\da$ is isomorphic
as an $A$--module to $A \oplus J$, it follows that
$$
{\rm
depth}_A(\da)={\rm depth}(A \oplus J)={\rm min}\{{\rm depth}(J),
{\rm depth}(A)\}={\rm depth}(J)
$$
and, therefore, $\da$ is a CM
$A$--module if and only if $J$ is a CM $A$--module of dimension equal to $\dim(A)$ 
(that is, if and only if $J$ is a maximal CM $A$-module).
\end{oss}

\begin{oss}
If $J$ is not finitely generated as $A$--module, it is more
problematic to find conditions implying $\da$ CM.  One can get
more information if the embedding $\iota:A\rightarrow \da$ is flat
(or, equivalently, if the $A$--module $J$ is flat). In this case,
$\da$ is CM if and only if both $A$ and $\da/M(\da)$ are CM
\cite[Theorem 2.1.7]{b-h}. As an example, set $A:=k[[X]]$,
$B:=k[[X,Y]]$ (where $k$ is a field), and let $J:= M:= (X,Y)$   be
the maximal ideal of $B$. Let $f: A \hookrightarrow B$ be the
inclusion. Clearly,  $J=\prod_{n\geq 1} f(A)Y^n$ is flat as an
$A$-module. Moreover, both $\da$, which is isomorphic to
$k[[X,Y,Z]]/(Y,Z)\cap(X-Y)$, and $\da/M(\da)$, which is isomorphic
to $k[[Y,Z]]/(Y^2,YZ)$, are not CM.
\end{oss}

In order to study when $\da$ is a Gorenstein ring, we need to look
at $A$ endowed with a natural structure of an $\da$--module.

The next  proposition  holds in general, without assuming the
additional hypotheses on $A$, stated at the beginning of the
section.

\begin{prop} \label{hom}
Preserve the notation of Proposition \ref{inizio}, and consider
the natural  map $\Lambda:f^{-1}(J)\rightarrow
\Hom_{A\Join^f\!J}(A,\da)$, where
$\Lambda(x):=\lambda_x:A\rightarrow\da$ is the $\da$--linear map
defined by $\lambda_x(a):=(ax,0)$, for each $a\in A$ and $x \in
f^{-1}(J)$. Then, $\Lambda$ is an $A$--linear embedding and
$\Lambda$ is surjective if and only if ${\rm
Ann}_{f(A)+J}(J)=(0)$.
\end{prop}
\begin{proof}
The fact that $\Lambda$ is an $A-$linear embedding is
straightforward.
 Assume ${\rm Ann}_{f(A)+J}(J)=(0)$.  Fix now a $\da-$linear
map $g:A\rightarrow \da$   and  the elements $a_0 \in A$ and $j_0
\in J$ such that $(a_0,f(a_0)+j_0)=g(1)$.  For each $j \in J$, by
definition, $(1,1+j)\cdot 1=1$, hence $ g(1)=g((1,1+j)\cdot
1)=(1,1+j)g(1)=(a_0,f(a_0)+j_0+j(f(a_0)+j_0))\,, $ and thus
$j(f(a_0)+j_0)=0$. This proves that $f(a_0)+j_0\in {\rm
Ann}_{f(A)+J}(J)$ and so, by hypothesis, $f(a_0)+j_0=0$.   In
particular,  $a_0\in f^{-1}(J)$ and
 $\Lambda(a_0)=\lambda_{a_0}=g$.
Conversely, assume that  $\Lambda$ is surjective, take an element
$f(a_0)+j_0\in {\rm Ann}_{f(A)+J}(J)$, with $a_0 \in A$ and $j_0
\in J$, and consider the map $\varphi: A \rightarrow \da$ defined
by  $\varphi(a):=(a,f(a))(a_0,f(a_0)+j_0)$,  for each $a \in A$.
Of course, $\varphi$ is a homomorphism of (additive) abelian
groups. Take now  two elements $x\in A$ and
$(\alpha,f(\alpha)+\beta)\in \da$. Since
$(\alpha,f(\alpha)+\beta)\cdot x = \alpha x$, then $
\varphi((\alpha,f(\alpha)+\beta)\cdot x)=\varphi(\alpha x)=
(\alpha x,f(\alpha x))(a_0,f(a_0)+j_0) $. On the other hand, we
have
$$
(\alpha,f(\alpha)+\beta)\varphi(x)=(\alpha,f(\alpha)+\beta)(x,f(x))(a_0,f(a_0)+j_0)=\varphi(\alpha x)
$$
where the last equality holds since $\beta(f(a_0)+j_0)=0$. Thus
$\varphi$ is an $\da$-linear map and, since $\Lambda$ is
surjective, there exists an element $z\in f^{-1}(J)$ such that
$\varphi=\lambda_z$. Therefore $ (a_0,f(a_0)+j_0)=\varphi(1)=
\lambda_z(1)=(z,0), $ that is $f(a_0)+j_0=0$.
\end{proof}

Now we are able to give a sufficient condition and a necessary
condition for the ring $\da$ to be Gorenstein.

\begin{oss}\label{pr:7.6}
We preserve the notation of Proposition \ref{inizio}.  If  $A$ is
a local Cohen-Macaulay ring, with maximal ideal $M$,  having a
canonical module isomorphic (as an $A$--module) to $J$, then $\da$
is Gorenstein. As a matter of fact,  $\iota:A\rightarrow \da$ is a
local ring embedding, since, $\iota^{-1}(M^{\prime_f})=M$. The
conclusion is a consequence of an unpublished result by Eisenbud
\cite[Theorem 12]{d'a} (see also \cite{shapiro}), applied to the
following short exact sequence of $A-$modules
$$
0\rightarrow A\stackrel{\iota}{\rightarrow}\da\rightarrow J\rightarrow 0\,.
$$
\end{oss}

\begin{prop}\label{pr:7.7}
We preserve the notation of Proposition \ref{inizio}. Assume that
$A$ is a local Cohen-Macaulay ring and that ${\rm Ann}_{f(A)+J}(J)=(0)$. If
$\da$ is Gorenstein, then $A$ has a canonical module isomorphic to
$f^{-1}(J)$.
\end{prop}
\noindent \textsc{Proof.} We begin by noting that, since $\da$ is
Gorenstein, it has a canonical module isomorphic to $\da$ as an
$\da-$module. Moreover, since the ring embedding $\iota$ is
finite, we have $\dim(A)=\dim(\da)$. Thus, keeping in mind that
$A$ is a cyclic $\da-$module (via the projection of $\da$ onto
$A$) and applying Proposition \ref{hom} and \cite[Theorem
21.15]{ei}, it follows that $A$ has a canonical module isomorphic
(as an $A-$module) to
$$
{\rm Ext}^0_{A\bowtie^f\!J}(A,\da)\cong \Hom_{A\bowtie^f\!J}(A,\da)\cong f^{-1}(J).
$$
The proof is now complete. \hfill$\Box$

\medskip

As a consequence of Remark \ref{pr:7.6} and Proposition
\ref{pr:7.7}, we deduce immedia\-tely the following.

\begin{cor} We preserve the notation of Proposition \ref{inizio}. Let $A$ be a local
Cohen-Macaulay ring having a canonical module isomorphic to $J$ as
an $A-$module and such that ${\rm Ann}_{f(A)+J}(J)=(0)$. Then,
$f^{-1}(J)$ and $J$ are isomorphic as $A$-modules.
\end{cor}

With extra assumptions on the ideal $f^{-1}(J)$ and on the ring $f(A)+J$, we can obtain the following characterization of when $\da$ is Gorenstein.
\begin{prop}
We preserve the notation and the assumptions of the beginning of the present section and, moreover, we assume that $A$ is a CM ring, $f(A)+J$ is (S$_1$) and equidimensional, $J\neq 0$ and that $f^{-1}(J)$ is a regular ideal of $A$. Then, the following conditions are equivalent.
\begin{enumerate}[\rm (i)]
\item $\da$ is Gorenstein.
\item $f(A)+J$ is a CM ring, $J$ is a canonical module of $f(A)+J$ and $f^{-1}(J)$ is a canonical module of $A$.
\end{enumerate}
\end{prop}
\begin{proof}
By Remark \ref{de}(e), $\da$ can be obtained as a fiber product of two surjective ring homomorphisms.
Then, the conclusion follows 
by applying \cite[Theorem 4]{ogoma-1987}.
\end{proof}


We conclude this section by comparing the multiplicity of $\da$
with the multiplicity of $A$. We assume the standing hypotheses of
the present section and  that $A$ is a local Cohen-Macaulay ring of Krull dimension
$n>0$. In particular, by Remark 3.3, if $I$ is an $M-$primary
ideal, then $I(\da)=I\!\Join^f\!\!\!(f(I)B)J$ (Proposition
\ref{extension}(2)) is $M^{\prime_f}-$primary. Furthermore, we
also assume that  $\da$ is a Cohen-Macaulay ring and that the
residue field $\boldsymbol k$ of $A$ and $\da$ is infinite.

Under these assumptions, we have that the multiplicity $e(A)$
of $A$ equals $\lambda_A(A/I)$, where $I$ is any minimal reduction
of $M$ \cite[Proposition 11.2.2]{hs} and where $\lambda_A(E)$
denotes the length of an $A-$module $E$. In particular, since $I$
 is a minimal reduction of $M$ and $A$ has infinite residue
field, it is minimally generated by $n$ elements (where $n=\dim
(A)=\dim(\da)$; see \cite[Lemma 8.3.7]{hs}); moreover,
$I=(a_1,a_2, \dots, a_n)$ is an $M-$primary ideal of a
Cohen-Macaulay local ring, hence it is generated by a regular
sequence. By \cite[Lemma 8.1.3]{hs}, $I(\da)$ is a reduction of
$M^{\prime_f}$ and, since the ideal $I(\da)=((a_1,f(a_1),
(a_2,f(a_2), \dots,(a_n,f(a_n))$ is generated by $n$ elements, it
is a minimal reduction \cite[Corollary 8.3.6]{hs}. Hence, the
multiplicity $e(\da)$ of $\da$ coincides with
$\lambda_{A\bowtie^f\!J}(\da/I(\da))$.

\begin{prop} \label{pr:7.9}
We preserve the notation of Proposition \ref{inizio}. Assume that
both $A$ and $\da$ are Cohen-Maculay local rings. Let $I$ be a
minimal reduction of $M$.  Then,
$e(\da)=e(A)+\lambda_{f(A)+J}(J/(f(I)B)J)$.
\end{prop}

\noindent \textsc{Proof.} By the previous observations, we know
that the equality $e(\da)= \lambda_{A\bowtie^f\!J}(\da/I(\da))$
holds. Moreover, we have
$$
\lambda_{A\bowtie^f\!J}(\da/I(\da))=\lambda_{A\bowtie^f\!J}
(\da/I\!\Join^f\!\!\!J)+\lambda_{A\bowtie^f\!J}(I\!\Join^f\!\!\!J/I(\da))\,.
$$
Now, since by Proposition \ref{inizio}(2) $A/I\cong
\da/I\!\!\Join^f\!\!\!J$ (as rings),  we have
$\lambda_{A\bowtie^f\!J}(\da/I\!\!\Join^f\!\!\!J)=\lambda_A(A/I)=e(A)$.
Moreover, again by Proposition \ref{inizio} (3), for every ideal
$L$ of $\da$ such that $I(\da)=I\!\!\Join^f\!\!\!(f(I)B)J\subseteq
L\subseteq I\!\Join^f\!\!\!J$, the image $p_B(L)$ is an ideal of
$f(A)+J$ such that $(f(I)B)J \subseteq p_B(L) \subseteq J$.
Conversely, for every ideal $H$ of $f(A)+J$ such that
$f(I)J\subseteq H\subseteq J$, then (by Proposition
\ref{extension}(1)) $I\!\Join^f\!\!H$ is an ideal of $\da$ such
that $I\!\Join^f\!\!\!(f(I)B)J \subseteq H\subseteq
I\!\Join^f\!\!\!J$. Hence,  we easily conclude that
$\lambda_{A\bowtie^f\!J}(I\!\Join^f\!\!\!J/I(\da))=\lambda_{f(A)+J}(J/(f(I)B)J)$
and the proof is complete. \hfill$\Box$\\

When $A=B$, and $f=\mbox{\it id}_A$, the amalgamation along $J$
gives rise to the amalgamated duplication $A\!\Join\! J$. In this
case we obtain a better result about the multiplicity.


\begin{cor} \label{cor:7.10}
Let $(A, M)$ be a Cohen-Macaulay local ring and $J$ be an ideal of
$A$ with $\dim_A(J)=\dim(A)$. Let $I$ be any minimal reduction of
$M$. Then $e(A\!\Join \!\! J)=e(A)+\lambda_A(J/IJ)$. In
particular, if $\dim(A)=1$, then $e(A\!\Join \!\! J)$ $=2e(A)$.
\end{cor}

\noindent \textsc{Proof.}  The first statement is a
straightforward consequence of the previous proposition. As for
the one-dimensional case, any minimal reduction $I$ of $M$ is
principal; hence $IJ=I\cap J$ and
$\lambda_A(J/IJ)=\lambda_A((I+J)/I) \leq \lambda_A(A/I)=e(A)$. On
the other hand, by \cite[Proposition 11.1.10 and Theorem
11.2.3]{hs}, $\lambda_A(J/IJ) \geq e(I;J)=e(M;J) \geq e(A)$ (where
$e(I;J)$ denotes the multiplicity of $I$ on the $A-$module $J$;
see \cite[Definition 11.1.5]{hs}). Hence, we have the equality
$\lambda_A(J/IJ)=e(A)$ and the proof is complete. \hfill$\Box$\\

\section{Appendix}
Let $f:A\longrightarrow B$ be a ring homomorphism and let $J$ be an ideal of $B$. 
By Corollary \ref{localefarf}(3), when $A$ is a local ring with maximal ideal  $M$ and $J$ is contained in the Jacobson radical of $B$, then $\am$ is a local ring with maximal ideal  $M^{\prime_f}:=\{(m,f(m)+j)\mid m\in M, j\in J \}$.
As it was proved in 
Proposition \ref{embdimineq}(1)
, if $\am$ is a local ring with finitely generated maximal ideal, then the maximal ideal  $M$ of $A$ is finitely generated and the following inequality ${\rm embdim}(A)\leq {\rm embdim}(\am)$ holds.
However, part 2 of 
Proposition \ref{embdimineq}
and Theorem \ref{embdimeq} hold under the additional assumption, not explicitly declared, that $B = f(A) +J$.
The following example shows that it is possible that  $B \supsetneq f(A) +J$ and $J$ is finitely generated as an ideal of $B$, but not finitely generated as an ideal of  $f(A) +J$. 
\smallskip

\begin{ex}
	Let $A:=K$ be a field and $T,U$ be indeterminates over $K$.  Set  $B:=K(U)[T]_{(T)}$ and $J:=TK(U)[T]_{(T)}$.
	By \cite[Example 2.6]{dafifoproc},  the integral domain $K+TK(U)[T]_{(T)}$ is canonically isomorphic to  $A \Join^f J$, where $f:A  \rightarrow B$ is the natural embedding.
	By Lemma \ref{fon} and Corollary \ref{localefarf}(3), $f(A)+J= K+TK(U)[T]_{(T)}$ is local and 1-dimensional and the prime spectrum of $f(A)+J$ coincides with that of the DVR $B$. 
	Since the field extension $K\subseteq K(U)$ is not finite, it is easy to infer that $f(A)+J$ is non Noetherian and thus its maximal ideal $J$, as an ideal  of $f(A)+J$,  is not finitely generated.
\end{ex}

If $B \neq f(A)+J$,  the correct assumption in
Proposition \ref{embdimineq}(2)
in order to ensure that  $M^{\prime_f}$ is finitely generated is to require that  $M$ is a finitely generated ideal of $A$ and $J$ is a finitely generated ideal of $f(A)+J$, as shown in the next result. 
\smallskip

\begin{prop}\label{corrected}
Let $f:A\longrightarrow B$ be a ring homomorphism and let $J$ be an ideal of $B$. Assume that $A$ is local with finitely generated maximal ideal  $M$ and that $J$ is finitely generated, as an ideal of $f(A)+J$, and that $J$ is contained in the Jacobson radical of $B$. Then, the ring $\am$ is local with finitely generated maximal ideal and moreover we have
$${\rm embdim}(\am)={\rm embdim}(A)+\nu(J)\,,$$
where now $\nu(J)$ denotes the minimum number of generators of $J$ as an ideal of the ring $f(A)+J$.
\end{prop}
\noindent
\textsc{Proof.} Let $\{m_1, m_2, \ldots, m_r\}$ (respectively, $\{j_1, j_2, \ldots, j_s \}$) be minimal sets of generators of $\f m$ (respectively, of $J$ as an ideal of $f(A)+J$). We now claim that 
$$
\mathcal G:=\{(m_i,f(m_i)), (0,j_h)\mid i=1, 2,\ldots, r, \,  h=1, 2, \ldots,s  \}
$$
is a minimal set of generators of  $M^{\prime_f}$. The fact that $\mathcal G$ generates  $M^{\prime_f}$ is straightforward and we left its easy proof to the reader. 
To prove that $\mathcal G$ is minimal with respect to the property of generating  $M^{\prime_f}$ it suffices to show that the canonical image of $\mathcal G$ into 
$ M^{\prime_f}/( M^{\prime_f})^2$ is linearly independent over the residue field $ \boldsymbol{k}$ of $A$. 
Let $a_1, a_2, \ldots,a_r,\, \alpha_1, \alpha_2,\ldots,\alpha_s\in A$ be such that
$$
\sum_{i=1}^r[a_i]_{ M}[(m_i,f(m_i))]_{ M^{\prime_{\!_f}}} + \sum_{h=1}^s[\alpha_h]_{ M}[(0, j_h)]_{{ M^{\prime_{\!_f}}}}=0 \qquad \mbox{in }\qquad { M^{\prime_f}}/({ M^{\prime_f}})^2 \qquad (\star).
$$

The same argument given in
Theorem \ref{embdimeq}
proves that $a_i\in  M$, for $i=1, 2, \ldots,r$, and thus $(\star)$ is equivalent to state that 
$$
\mathbf x:=\sum_{h=1}^s(0,f(\alpha_h)j_h)\in ( M^{\prime_f})^2 . 
$$
By definition, $\mathbf x$ is sum of elements of the type $(\mu_k,f(\mu_k)+u_k)(\mu_k',f(\mu_k')+u_k')$, for $k=1,2, \ldots, t$, with $\mu_k,\mu_k'\in  M$ and $ u_k,u_k'\in J$. 
It follows that $\sum_{k=1}^t\mu_k\mu_k'=0$, and then $\sum_{h=1}^sf(\alpha_h)j_h\in f( M)J+J^2\subseteq J(f( M)+J)$. 
By contradiction, assume that there exists some index $h$ such that $\alpha_h \in A\setminus  M$. 
{ Say $h=1$,}  let $\lambda_1$ be the inverse of $\alpha_1$ in $A$. Then $f(\lambda_1)\sum_{h=1}^sf(\alpha_h)j_h\in J(f( M)+J)$. 
Take elements $\eta_1, \eta_2,\ldots\eta_s\in M$ and $v_1, v_2,\ldots,v_s\in J$ such that 
$$
j_1+f(\lambda_1)\sum_{h=2}^sf(\alpha_h)j_h=f(\lambda_1)\sum_{h=1}^sf(\alpha_h)j_h=\sum_{h=1}^s(f(\eta_h)+v_h)j_h\,.
$$
It follows that $j_1(1-f(\eta_1)-v_1)\in (j_2, j_3, \ldots,j_s)(f(A)+J)$. 
Since $f( M)+J$ is the maximal ideal of the local ring $f(A)+J$, it follows that $1- f(\eta_1)-v_1$ is invertible in $f(A)+J$, that is, $j_1\in (j_2, j_3, \ldots,j_s)(f(A)+J)$, contradicting the minimality of $\{j_1, j_2, \ldots, j_s \}$. The proof is now complete. \hfill$\Box$

\smallskip

\begin{oss}
	Note that if $J$ is finitely generated as an $A$-module (with the structure induced by the ring homomorphism $f$), then it is finitely generated as an ideal of $f(A)+J$ too, as it is easily seen. The converse is not true, by \cite[Remark 5.10]{dafifoproc}.
\end{oss}  

\begin{oss}
If $A$ is local with finitely generated maximal ideal $ M$ such that $f( M)B=B$ and $J$ is finitely generated as an ideal of (the local ring) $f(A)+J$, then Nakayama's Lemma implies that $J=0$, according to Propositions \ref{edimlower} and \ref{corrected}. 
\end{oss}

\noindent
\textsc{Question.} Is there a local amalgamation $\am$ with finitely generated maximal ideal    such that $J$ is not finitely generated as an ideal of $f(A)+J$ and $f( M)B\neq B$ (where $ M$ is the maximal ideal of $A$)? 


\begin{thebibliography}{9999}


\footnotesize

\bibitem{aam} H. Ananthnarayan, L.L. Avramov and W.F. Moore {\em
Connected sums of Gorenstein local rings}, J. reine angew. Math.
{\bf 667} (2012), 149--176; DOI: 10.1515/crelle.2011.132.


\vskip -1cm\bibitem{a-06}
D.D. Anderson,  \emph{Commutative rngs}, in {\sl ``Multiplicative Ideal
Theory in Commutative Algebra: A tribute to the work of Robert
Gilmer''} (Jim Brewer, Sarah Glaz, William Heinzer, and Bruce
Olberding Editors), 1--20,  Springer, New York, 2006.
%
%
%
%
\bibitem{am} M.F. Atiyah and I.G. MacDonald,
Introduction to commutative algebra, Addison-Wesley, Reading, 1969.

\bibitem{aung}
Pye Phyo Aung, {\em Gorenstein dimensions over some rings of the form $R \oplus C$}, (2014), arXiv:1408.1123.


\bibitem{basataya} A. Bagheri, M. Salimi, E. Tavasoli, S.
Yassemi, {\em A construction of quasi-Gorenstein rings}, J.
Algebra. Appl. {\bf 11} (2012) 1250013-1--1250013-9; DOI:
10.1142/S0219498811005361.
%


\bibitem{bo} M.B. Boisen and P.B. Sheldon,
\emph{CPI--extension: overrings of integral domains with special prime spectrum},
Canad. J. Math. \bf 29 \rm (1977),  722--737.


\bibitem{b-h} W. Bruns, J. Herzog, Cohen-Macaulay rings, Cambridge University Press, Cambridge,
1993.

\bibitem{c-j-k-m}
M Chhiti, M. Jarrar, S. Kabbaj, and N. Mahdou, {\em Pr\"ufer
conditions in an amalgamated duplication of a ring along an
ideal}, Comm. Algebra {\bf 41} (2015), 249--261.

\bibitem{d'a} M.~D'Anna,  {\em  A construction of Gorenstein
rings}, \rm J. Algebra \bf 306 \rm (2006), 507--519.

\bibitem{dafifoproc} M. D'Anna, C.A. Finocchiaro and  M. Fontana
{\em Amalgamated algebras along an ideal}, in ``Commutative
Algebra and Applications'', Proceedings of the Fifth International
Fez Conference on Commutative Algebra and Applications, Fez,
Morocco, 2008, W. de Gruyter Publisher, Berlin, 2009.

\bibitem{dafifoJPAA} M. D'Anna, C.A. Finocchiaro and M. Fontana
{\em Properties of chains in an amalgamated algebra along an
ideal}, J.  Pure  Appl. Algebra  {\bf 214} (2010), 1633--1641.


\bibitem{d'a-f-1} M.~D'Anna and M. Fontana,
{\em An amalgamated duplication of a ring  along an ideal: the
basic properties},  J. Algebra Appl.   \bf 6 \rm  (2007), 443--459.


\bibitem{d'a-f-2} M.~D'Anna and M. Fontana,
\emph{The amalgamated duplication of a ring along a
multiplicative-canonical ideal}, Arkiv Mat. {\bf 45} (2007),
241--252.




\bibitem{do1} J.L. Dorroh, \emph{Concerning adjunctions to algebras},
Bull. Amer. Math. Soc. \textbf{38} (1932),  85--88.


\bibitem{ei} D. Eisenbud, Commutative Algebra with a view toward Algebraic Geometry, Springer Verlag, New York, 1995.

\bibitem{fa}
A. Facchini, {\em Fiber products and Morita duality for
commutative rings,} Rend. Sem. Mat. Univ. Padova {\bf 67} (1982),
143--159.

\bibitem{F}  M. Fontana, {\em Topologically defined classes of
commutative rings}, Ann. Mat. Pura Appl.  \bf 123 \rm (1980),
331--355.
%
\bibitem{Fos} R. Fossum, {\em Commutative extensions by canonical modules are Gorenstein rings},
Proc. Am. Math. Soc. {\bf 40} (1973), 395--400.


\bibitem{H} J. Huckaba, Commutative rings with zero divisors, M.
Dekker, New York, 1988.
%
\bibitem{m-y} H. R. Maimani and S. Yassemi
\emph{Zero-divisor graphs of amalgamated duplication of a ring along an ideal}, J. Pure Appl. Algebra, {\bf
212} (2008), 168--174.

\bibitem{hs} C. Huneke and I. Swanson,  Integral Closure of Ideals,
Rings and Modules, London Math. Soc., Lecture Notes Ser. {\bf
336}, Cambridge University Press, Cambridge, 2006.

\bibitem{N} M. Nagata,  Local Rings, Interscience, New York, 1962.


\bibitem{ogoma}
T. Ogoma, {\em  Fibre products of Noetherian rings and their
applications,}  Math. Proc. Cambridge Philos. Soc. {\bf 97}
(1985), 231--241.

\bibitem{ogoma-1987} T. Ogoma, {\em Fiber products of Noetherian
rings}, Adv. Stud. Pure Math.  {\bf 11},  {\sl ``Commutative
Algebra and Combinatorics''} (Edited by M. Nagata and H.
Matsumura) (1987), 173--182.



\bibitem{sataya} M. Salimi, E. Tavasoli, S. Yassemi, {\em The
amalgamated duplication of a ring along a semidualizing ideal},
Rend. Sem. Mat. Univ. Padova  {\bf 129} (2013), 115--127; DOI:
10.4171/RSMUP/129-8.


\bibitem{shapiro} J. Shapiro, {\em  On a construction of Gorenstein rings proposed by M. D'Anna}, J. Algebra, {\bf 323} (2010), 1155--1158.


\end{thebibliography}
\end{document}